\documentclass[12pt, reqno]{amsart}
\usepackage{mathtools}
\usepackage{amssymb,amsmath,dsfont}
\usepackage[table,xcdraw]{xcolor}
\usepackage{pgfplots}
\pgfplotsset{compat=1.14}
\usepackage{enumitem}
\usepackage{calc}
\usepackage{hyperref}
\usepackage{comment}
\usepackage[clock]{ifsym}

\usepackage[margin=1in]{geometry}

\usepackage{booktabs}
\usepackage{caption} 
\newtheorem{thm}{Theorem}[section]

\newtheorem{lem}[thm]{Lemma}
\newtheorem{prop}[thm]{Proposition}

\newtheorem{claim}[thm]{Claim}

\theoremstyle{plain}

\DeclareMathOperator{\var}{Var}

\newcommand{\osc}{\mathrm{osc}}
\newcommand{\EE}{\mathbb{E}}
\newcommand{\PP}{\mathbb{P}}

\newcommand{\eps}{\varepsilon}

\newtheorem*{assumption*}{\assumptionnumber}
\providecommand{\assumptionnumber}{}


\newtheorem*{disorder*}{\assumptionnumber}
\providecommand{\assumptionnumber}{}


\title[The edge-averaging process on graphs with random initial opinions]{The edge-averaging process on graphs \\ with random initial opinions}

\author{Dor Elboim}
\address{Dor Elboim\hfill\break
    Department of mathematics, Stanford University, California, USA.}
\email{dorelboim@gmail.com}
\author{Yuval Peres}
\address{Yuval Peres\hfill\break
    Beijing Institute of Mathematical Sciences and Applications (BIMSA), Huairou district, Beijing, China.}
\email{yperes@gmail.com}
\author{Ron Peretz}
\address{Ron Peretz\hfill\break Economics Department, Bar Ilan University, Israel.}
\email{ron.peretz@biu.ac.il}
\date{\today}

\begin{document}

\begin{abstract}
In several settings (e.g., sensor networks and social networks), nodes of a graph are equipped with initial opinions, and the goal is to estimate the average of these opinions using local operations. A natural algorithm to achieve this is the {\em edge-averaging process}, where edges are repeatedly selected at random (according to independent Poisson clocks) and the opinions on the nodes of each selected edge are replaced by their average.  
The effectiveness of this algorithm is determined by its convergence rate. It is known that on a finite graph of $n$ nodes, the opinions  reach approximate consensus in polynomial time.   
We prove that the convergence is much faster when the initial opinions are disordered (independent identically distributed):
the time to reach approximate consensus is $O (\log^2n)$, and this bound is sharp. For infinite graphs, we show that for every $p\geq 1$, if the initial opinions are in $L^p$, then the opinion at each vertex converges to the mean in $L^p$, and if $p>4$, then almost sure convergence holds as well.
\end{abstract}

\maketitle

\begin{figure}[h]
     \centering
    \includegraphics[width=0.23\linewidth]{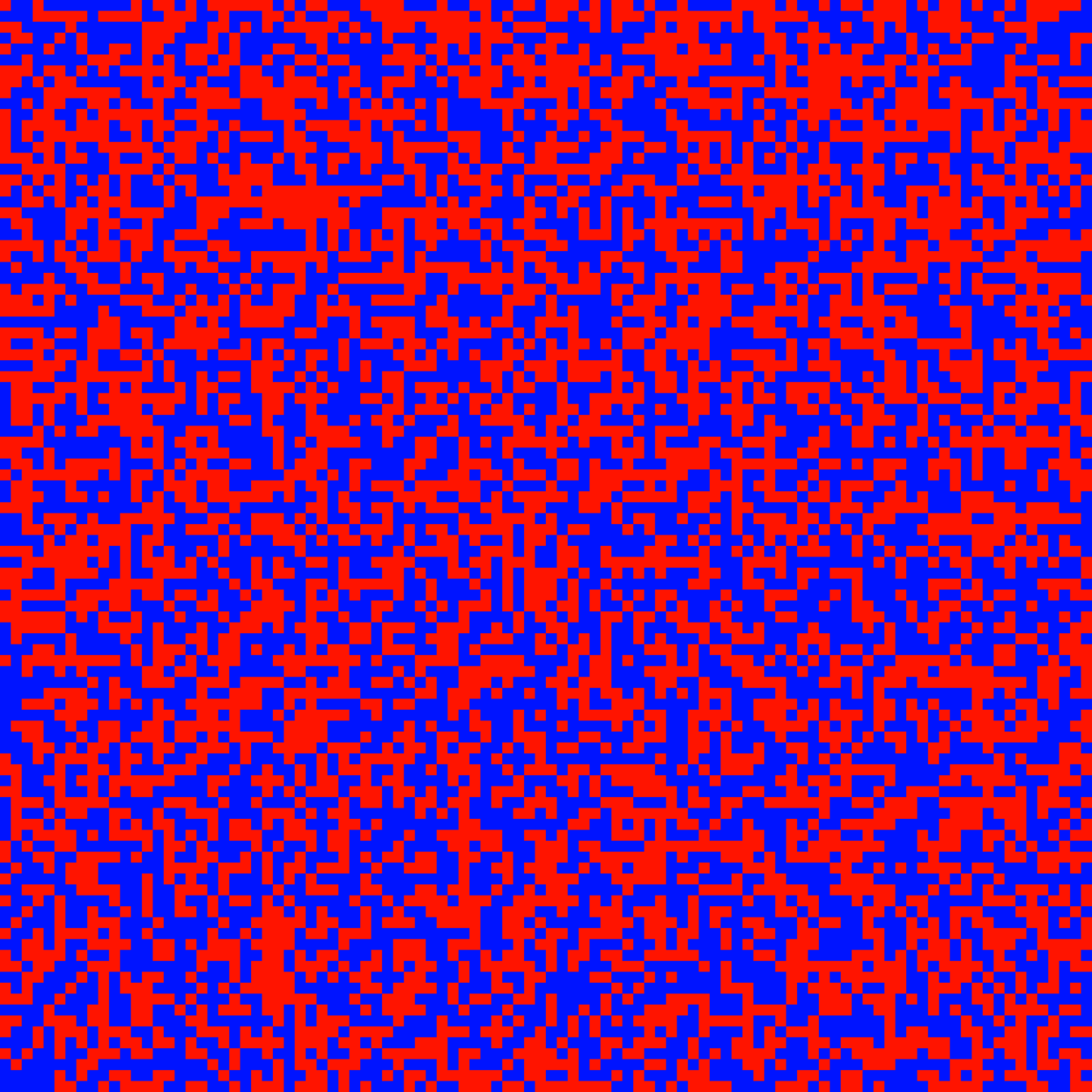}
\includegraphics[width=0.23\linewidth]{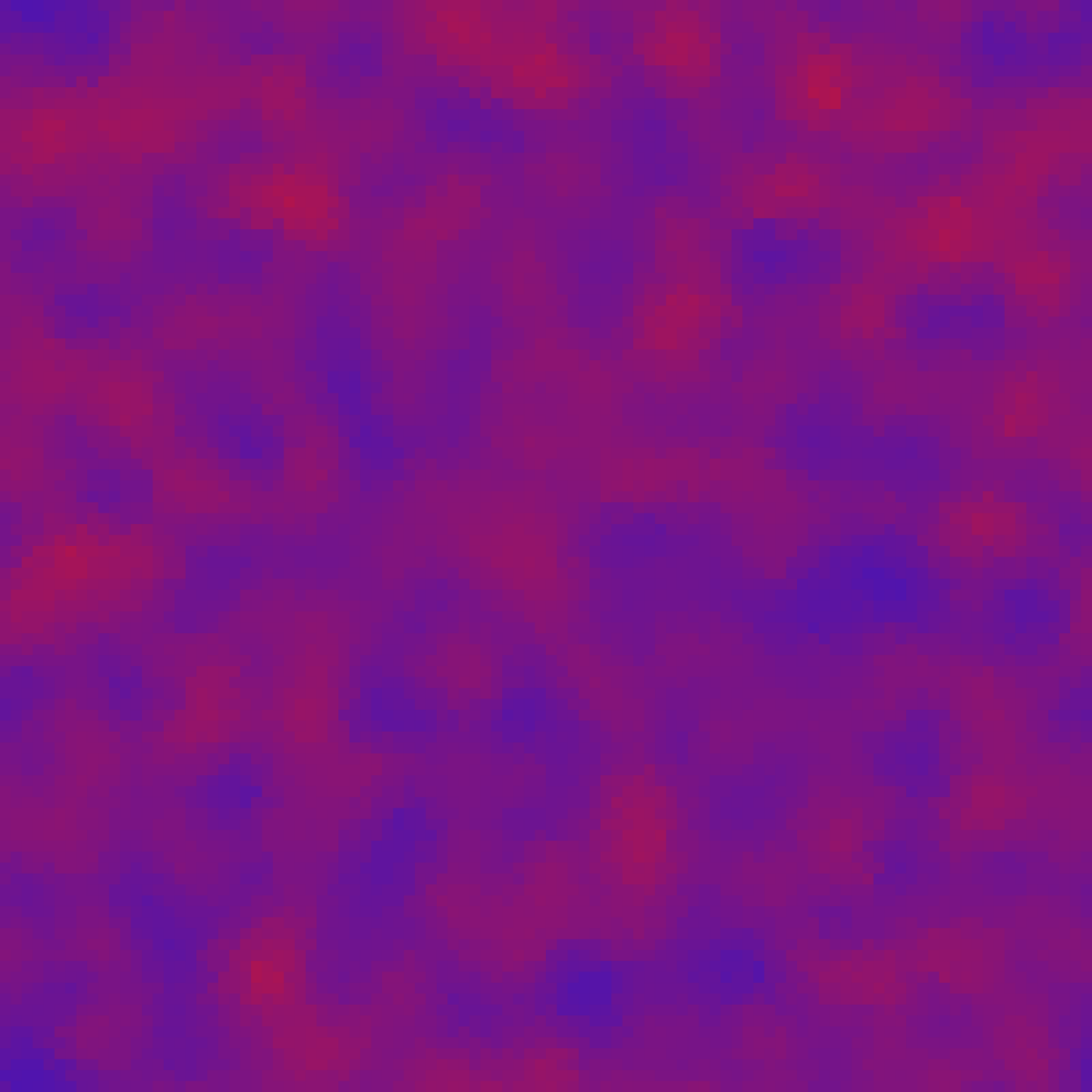}
    \includegraphics[width=0.23\linewidth]{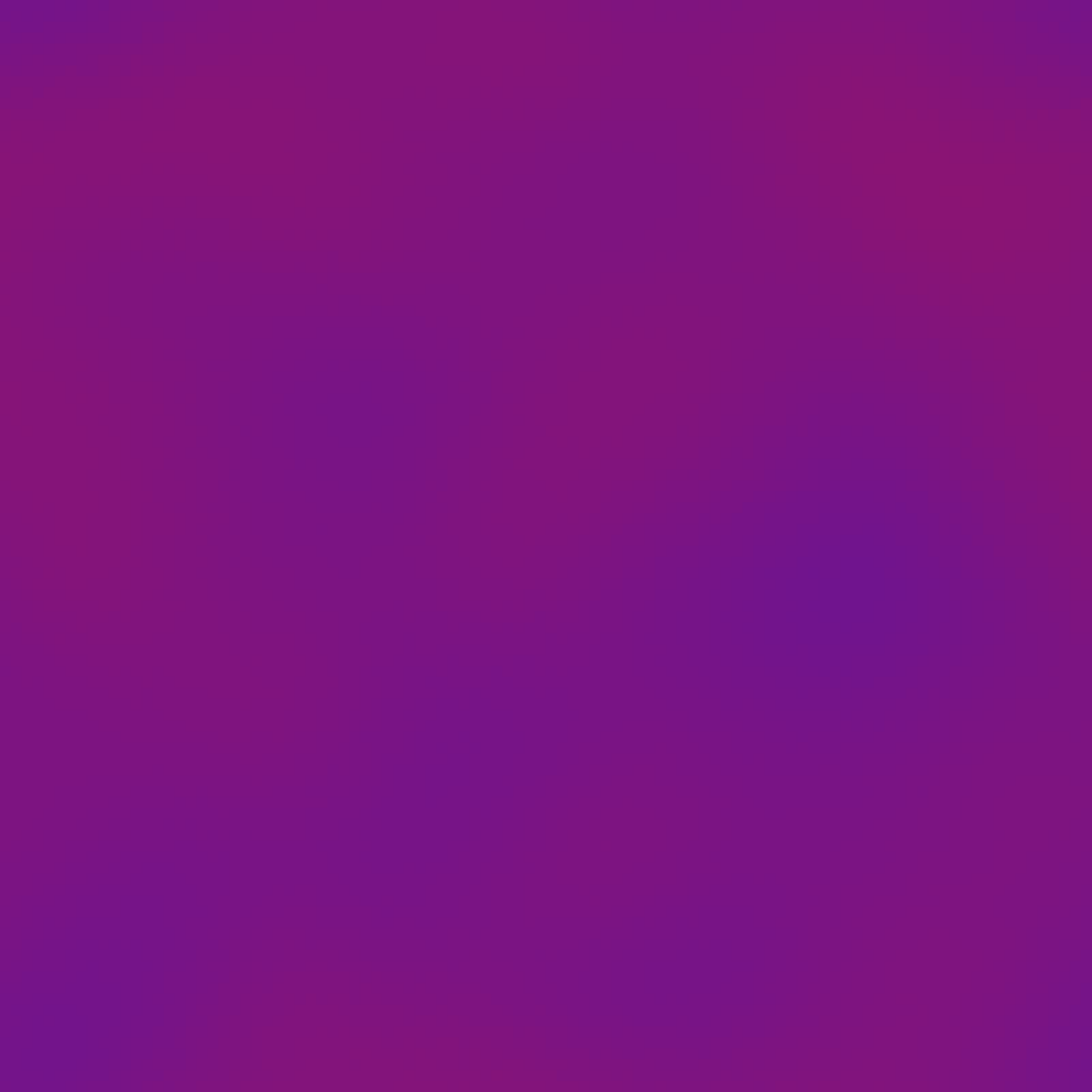}
    \includegraphics[width=0.23\linewidth]{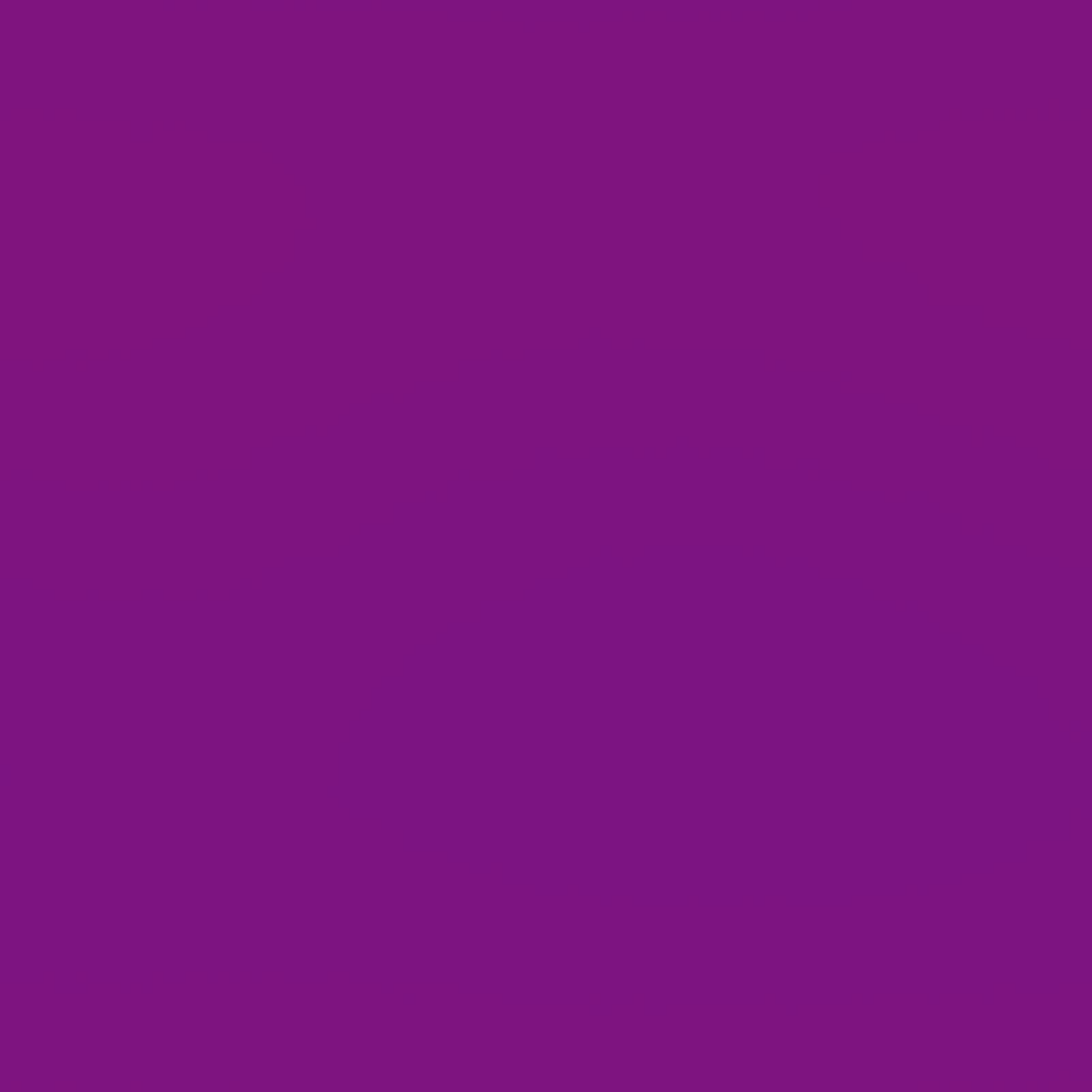}\\
    \vspace{0.2cm}
 \includegraphics[width=0.23\linewidth]{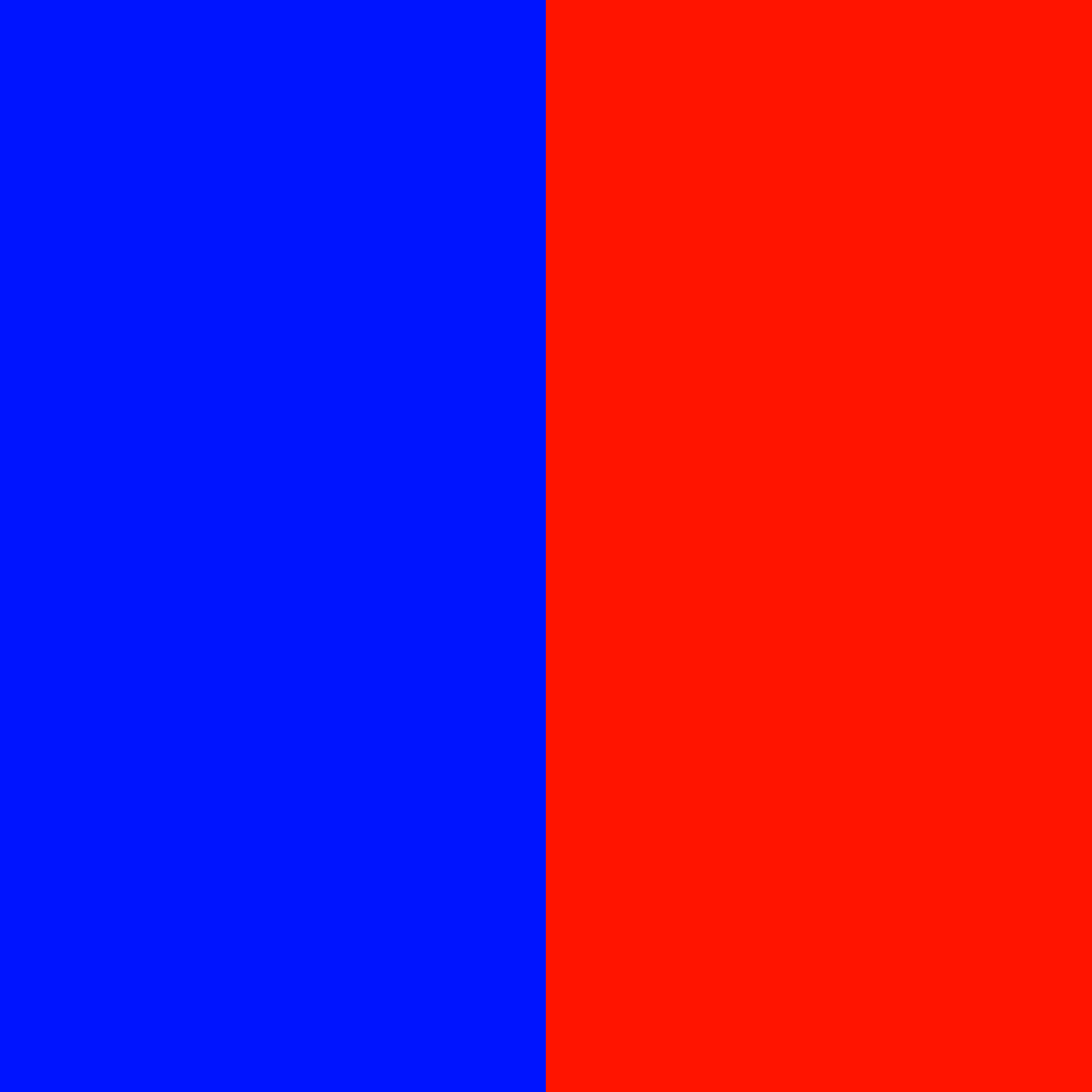}
    \includegraphics[width=0.23\linewidth]{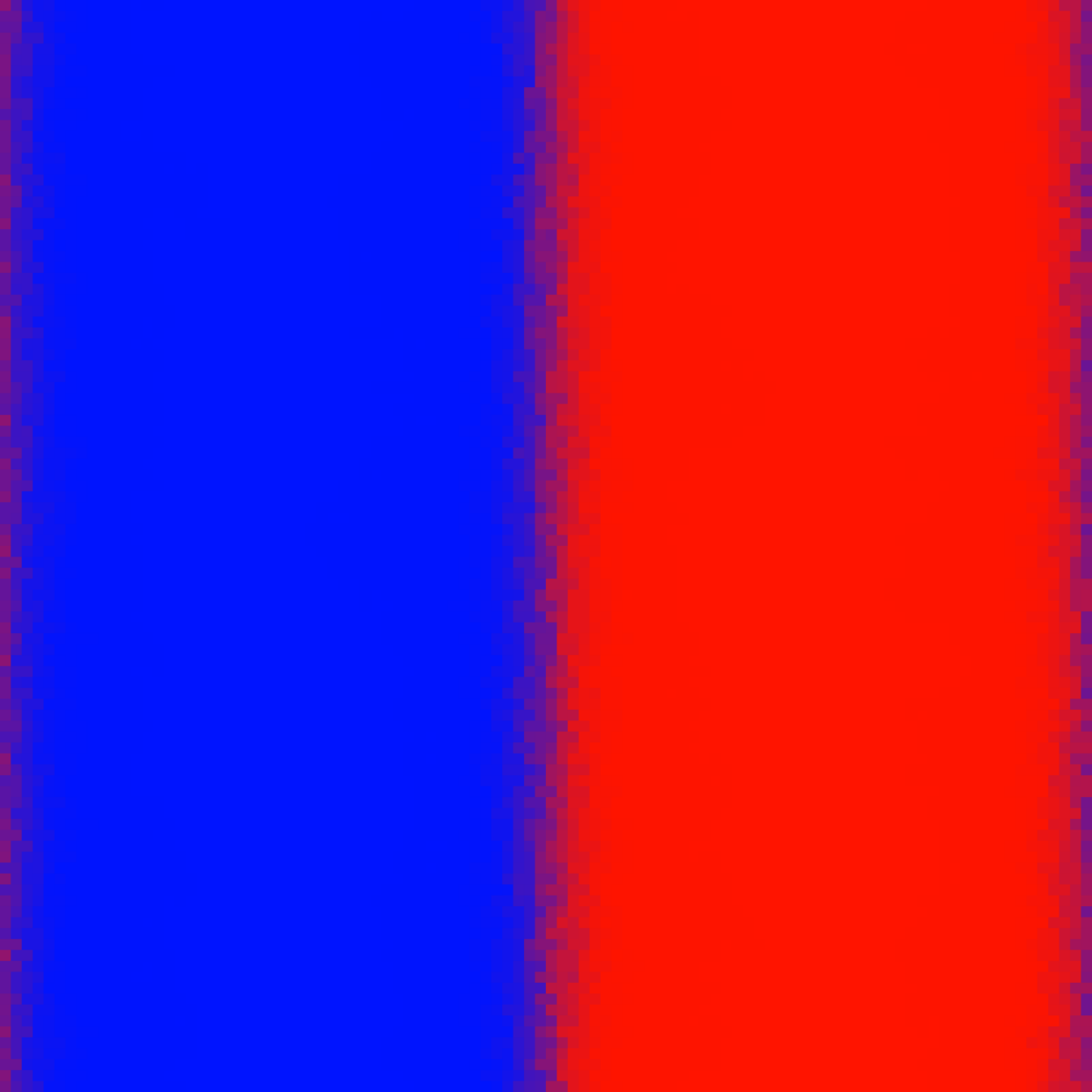}
    \includegraphics[width=0.23\linewidth]{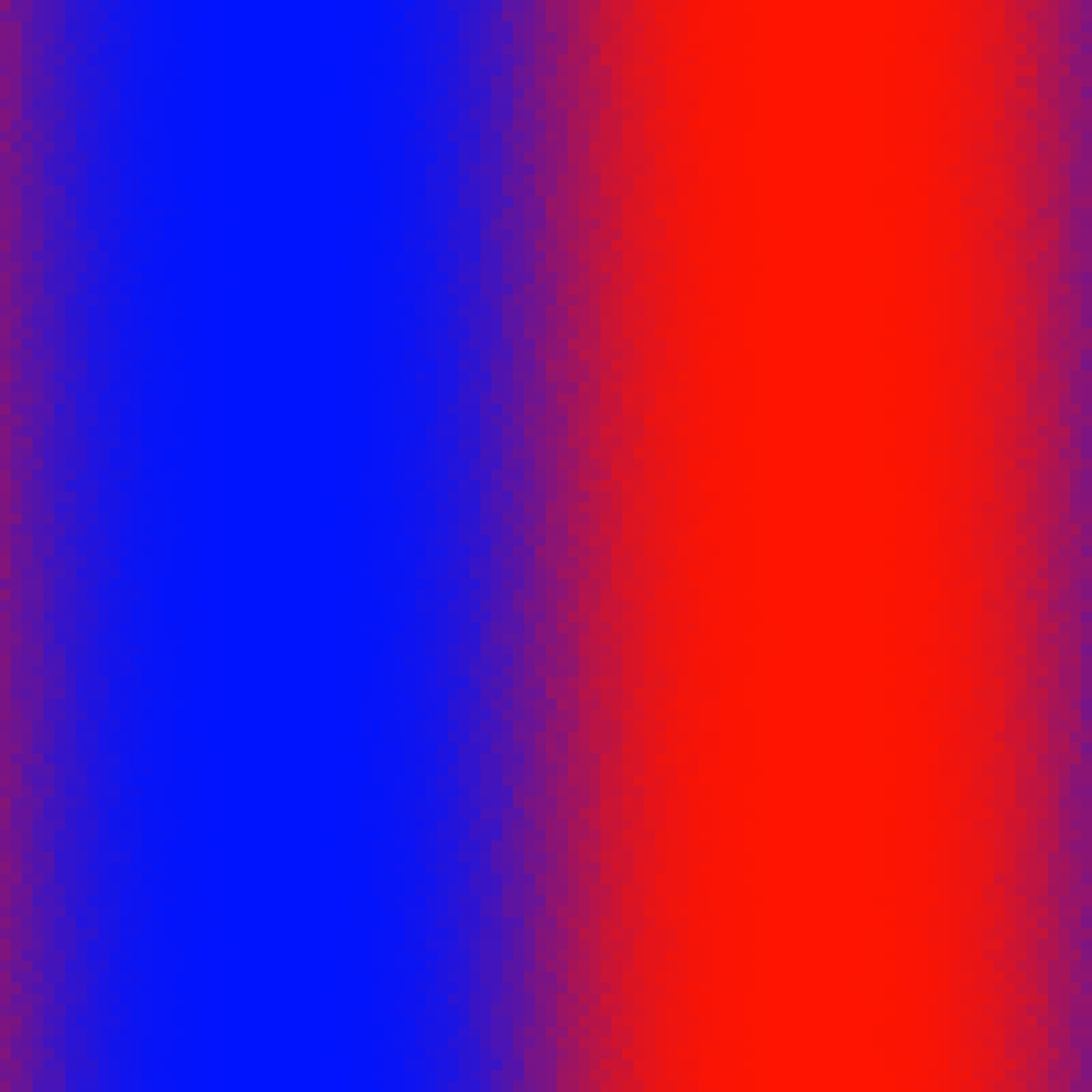}
    \includegraphics[width=0.23\linewidth]{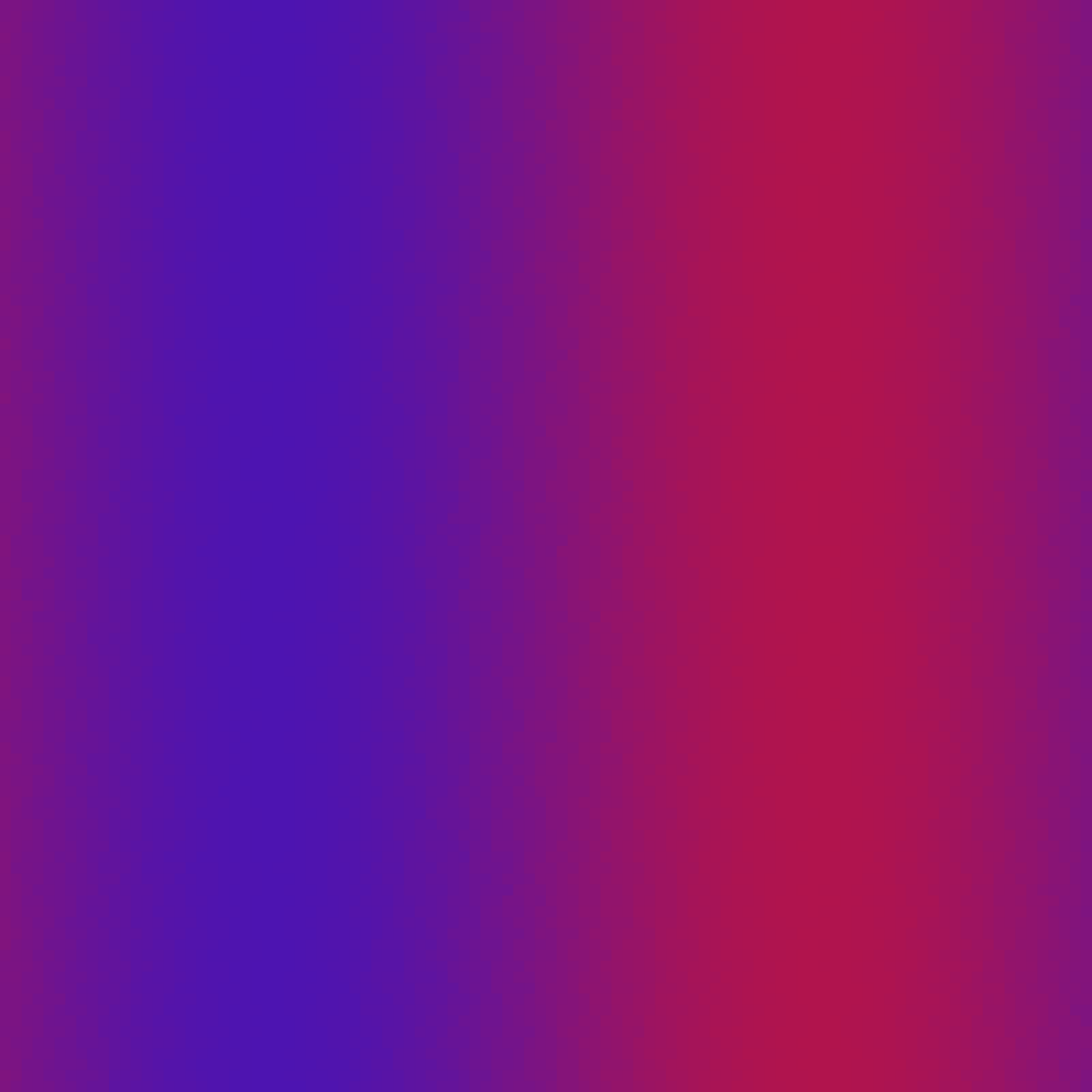}
    \caption{Simulation of the edge-averaging process on the $100\times 100$ discrete torus with uniform random (top) and deterministic (bottom) binary initial opinions. Snapshots taken at times (left to right) $0$, $7$, $60$, and $600$.}
    \label{fig:simulation}
\end{figure}

\newpage

\section{Introduction}

In a variety of networks, information at individual nodes (e.g., opinions in social networks \cite{golub2010naive} and measurements in sensor networks \cite{boyd2006randomized}) must be shared via local connections to reach a joint estimate. This situation is known in economics as distilling the ``wisdom of the crowd''.   Several iterated averaging processes have been used for this purpose; here we focus on the {\em edge-averaging process}; see the survey \cite{MR2908618}. Every node is equipped with an initial opinion (a real number). In each step, two adjacent nodes are selected at random, and their opinions are both replaced by their average. 

Since the maximal and minimal opinions are monotonic functions of time and the overall average opinion is an invariant of the system, it is not hard to show that all the opinions converge to that number. This desirable feature is often referred to as ``social learning;'' see  the survey by Golub and Sadler~\cite{Sadler2016-SADLIS}.

A natural question is how fast the opinions get close to the average. 
In their seminal paper, Boyd, Ghosh, Prabhakar, and Shah \cite{boyd2006randomized} addressed that question. 
They provided  spectral upper bounds for the number of updates required to approach consensus which depend only on the network topology, i.e., they hold  for all initial opinion profiles. These upper bounds are often a super-linear power law in the number of edges $m$, e.g., for a cycle of $m$ nodes they are of  order $m^3$. Nevertheless, for the worst case initial profiles the bounds of \cite{boyd2006randomized} are sharp.

\bigskip

It turns out that for typical initial opinion profiles, these worst-case upper bounds are too pessimistic. Figure~\ref{fig:simulation} indicates that convergence to consensus is much faster when the initial opinions are disordered.
We focus on initial 
opinion profiles that are independent and identically distributed (i.i.d.).  There are several motivations for this. First, taking these initial profiles to be i.i.d.\ uniform in $[0,1]$, high probability results correspond to  most (in terms of volume) initial profiles in $[0,1]^V$, where $V$ is the set of nodes.
Second, in sensor networks, the measurements at different nodes  often arise from independent noise added to a common parameter. 
Finally, the i.i.d.\ assumption leads to tractable problems, yet we believe that the results will apply to a wider class of initial opinion profiles satisfying some correlation decay hypothesis; see \ref{section:future}.

\subsection{The model} 
A network is modeled as a  connected graph $G=(V,E)$ which is finite, or infinite of bounded degree\footnote{The assumption of bounded degree can be relaxed to {\em stochastic completeness}, see Section \ref{section:existence}.}. Each vertex $v\in V$ has an opinion $f_t(v)\in \mathbb R$ at every time $t\ge 0$, and the opinions evolve over time as follows. Each edge is equipped with an independent unit-rate Poisson clock. When the clock of an edge $uv\in E$ rings at time $t$, both $u$ and $v$ update  their opinion to be the average of the two. That is,
\begin{equation}
\label{equation definition}    
f_t(u)=f_t(v)=\frac{f_{t-}(u)+f_{t-}(v)}{2},
\end{equation}
where $f_{t-}(u):=\lim _{s\uparrow t} f_s(u)$.
In this process, the number of updates per edge up to time $t$ has mean $t$ and is concentrated near $t$ for $t$ large. 
\subsection{Global convergence}
 For $f\colon V\to \mathbb R$, define its {\em oscillation} by 
 $$\osc(f):=\sup_{u,v\in V}|f(u)-f(v)|\,.$$ 
Typically, the opinions at different nodes never agree exactly, and there are several ways to define approximate consensus. Here, as in \cite{dolev1986reaching} and \cite{elboim2024asynchronous}, we require that all opinion differences be uniformly small. Formally,   given $\eps>0$, define the $\eps$-{\em consensus time} by $$\tau_{\eps}:=\min\{t\geq 0: \osc(f_t)\leq \eps\}\,.$$ 
 Figure \ref{fig:charts} depicts the function  $t\mapsto\osc(f_t)$ for the two simulations shown in Figure \ref{fig:simulation}.
 
 Our first theorem gives a polylogarithmic upper bound for the $\eps$-consensus time.
\begin{thm}\label{theorem finite}
    Consider the edge-averaging process on a graph $G=(V,E)$ with  $|V|=n$. Suppose that the initial opinions $\{f_0(v)\}_{v\in V}$ are i.i.d.\ random variables that satisfy $|f_0(v)|\leq 1$ a.s.  Then for every $\eps >0$, we have $$\EE[\tau_\eps]\le C\log^2(n)/\eps^4,$$ where $C>0$ is a universal constant. 
\end{thm}

The bound in Theorem~\ref{theorem finite} cannot be improved for general graphs. Indeed, it is tight for the cycle of length $n$; see Theorem~\ref{theorem tightness}. 

\begin{figure}[h]
    \centering
\includegraphics[width=1.0\linewidth,clip]{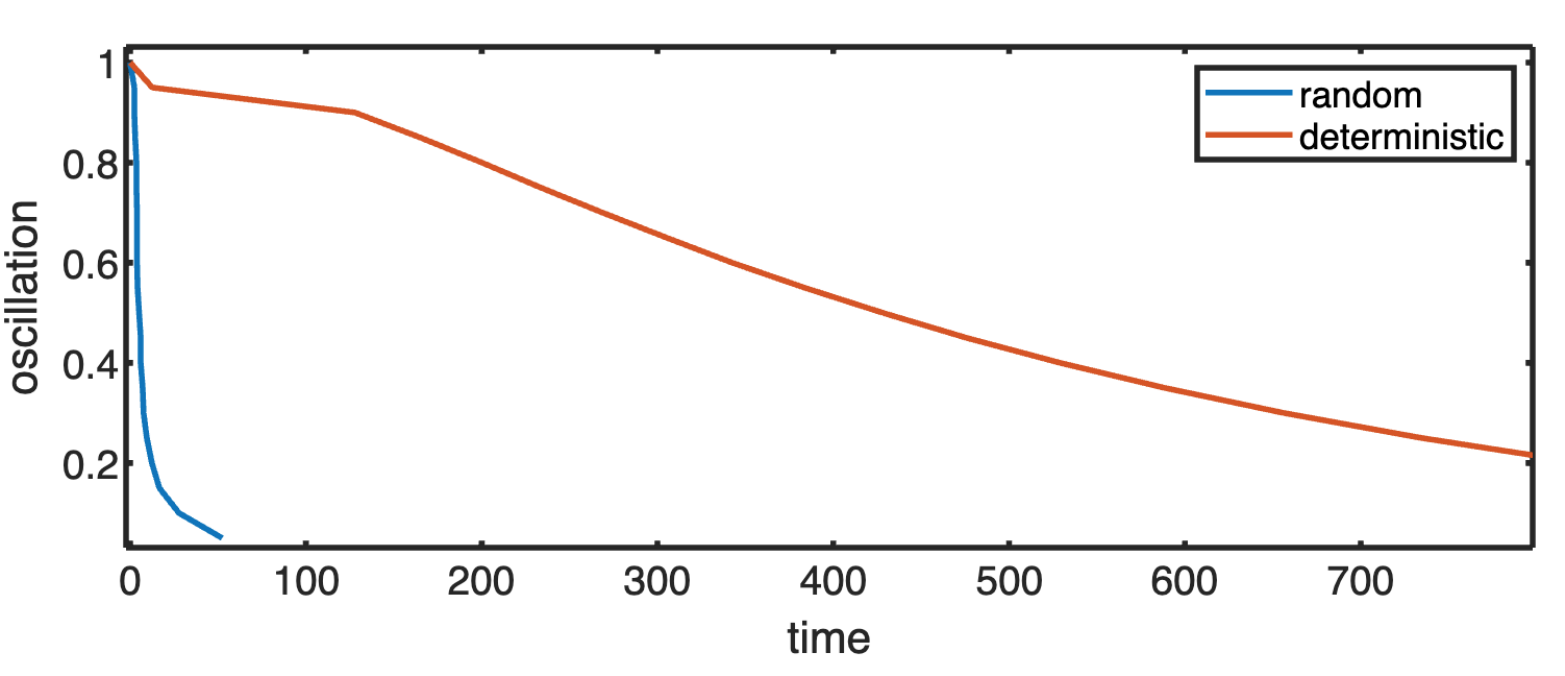}    
    \caption{A simulation of the edge-averaging process on the $100\times 100$ discrete torus. The graphs show the oscillation of the opinions $  \osc(f_t)$ as a function of time for random (blue) and deterministic (orange) initial opinions.}
    \label{fig:charts}
\end{figure}
\subsection{Local convergence}
Theorem~\ref{theorem finite} takes a global view of the system, referring to the oscillation of the opinions over the entire network.
When the network is massive (or infinite), it is more natural to study the convergence of the opinion at a fixed vertex.  This local point of view allows us to obtain bounds on convergence rate that do not depend on the size and structure of the network. The next result, which applies to both finite and infinite graphs,  is a local counterpart to Theorem \ref{theorem finite}, and is also the key to its proof. 

Throughout the paper, except Proposition~\ref{prop:arbitraryf_0}, the initial opinions are assumed to be i.i.d.\ random variables and their expectation is denoted by $\mu:=\EE[f_0(v)]$. 
\begin{thm}\label{theorem local bounded}
   Suppose that the initial opinions $\{f_0(v)\}_{v \in V}$ are i.i.d.\ and take values in $[-1,1]$. Fix $o\in V$. Then for every $t\ge 0$   and $\eps <1$, we have 
   \begin{equation}
    \mathbb P \big( |f_t(o)-\mu |\ge \eps  \big) \le 3 \exp \Big( -\frac{\eps ^2 \sqrt{t_*}}{12} \Big)\,,
\end{equation}
where  $t_*:=\min(t,|V|^2)$ (so $t_*=t$ if $V$ is infinite).
\end{thm}

This theorem is tight for the cycle, as we show in Proposition~\ref{prop:lowlocal}. 

Next, we consider unbounded initial opinions. Gantert and Vilkas \cite{gantert2024averaging} showed that if $V$ is infinite and the i.i.d.\ initial opinions $\{f_0(v)\}$ are in $L^2$, then $f_t(v) \to \mu$ in $L^2$ for every $v\in V$. 
The following theorem extends their result to $L^p$ for every $p\ge 1$, and gives a sharp rate of convergence.
We state it in a form that applies to both finite and infinite graphs.

\begin{thm}\label{theorem:Lp}
For every $p\in[1,\infty)$, there exists $C_p>0$ (depending only on $p$), such that if the i.i.d.\ initial opinions $\{f_0(v)\}_{v \in V}$ are in $L^p$, then for every $t>0$ and every $o\in V$,
\begin{equation}\label{eq:p}
\|f_t(o)-\mu\|_p\leq C _p\|f_0(o)-\mu\|_p \cdot
\begin{cases}
    t_*^{\frac{1-p}{2p}}&p\in [1,2),\\
    t_*^{-1/4}&p\geq 2,
\end{cases}
\end{equation}
where  $t_*:=\min(t,|V|^2)$. 
Moreover, if $V$ is infinite and  $p \in [1,2)$, then
$\|f_t(o)-\mu\|_p =o\big(t^{\frac{1-p}{2p}}\big)$ as $t \to \infty$. \newline (In particular $f_t(o) \to \mu$ in $L^1$ if $f_0(o) \in L^1$.)
\end{thm}

Proposition~\ref{theorem:lowerboundLp} shows that these upper bounds are essentially tight. In the setting of the preceding theorem, the triangle inequality yields that
$\|f_t(o)-f_t(v)\|_p \to 0$ for any two fixed nodes $o,v$; this is another natural sense of local convergence. Similar inferences can be made from Theorems \ref{theorem local bounded} and \ref{theorem infinite}.

Gantert and Vilkas~\cite{gantert2024averaging} conjectured that $f_t(v)\to \mu $ almost surely when the initial opinions are in $L^1$.
Our main result for infinite graphs is a proof of their conjecture under a stronger moment assumption. 

 \begin{thm}\label{theorem infinite}
    Suppose that $V$ is infinite, and for some $p>4$, the i.i.d.\ initial opinions $\{f_0(v)\}_{v \in V}$ are in $L^p$. Then for every $v\in V$, we have $\lim_{t\to\infty}f_t(v)=\mu$ almost surely. 
\end{thm}

\subsection{Other related work}
Aldous and  Lanoue~\cite{MR2908618}
described the basic properties of the edge-averaging process.
Chatterjee, Diaconis, Sly, and Zhang \cite{MR4385355} found the exact convergence time for the important special case of a fully connected network with arbitrary bounded initial opinions. Their result was followed by similar results for other networks (see Caputo, Quattropani, and Sau \cite{caputo2023cutoff}, and Movassagh, Szegedy, and Wang \cite{movassagh2024repeated}).

An important ingredient in the proof of $L^2$ convergence by Gantert and Vilkas~\cite{gantert2024averaging} mentioned before Theorem \ref{theorem:Lp}, is a recent result by Gollin
et al.~\cite{gollin2024sharing}, who established a general bound on dispersion of deterministic fragmentation processes.
Gantert, Hydenreich, and Hirscher \cite{gantert2020strictly} study the edge-averaging process in a setting where opinions take values in the unit circle instead of the real line. They showed that on the infinite path graph $\mathbb Z$,  the opinion at each vertex does not converge in probability, but the distance between opinions at adjacent vertices tends to zero in probability.

Golub and Jackson~\cite{golub2010naive}  consider a variety of averaging processes, in the setting where agents receive independent
noisy signals about the true value of a variable and then communicate in a network. A key insight of that paper is that all opinions
in a large society converge to the truth if and only if the influence
of the most influential agent vanishes as the society grows.

\section{Fragmentation}
For a vertex $o\in V$, consider a process $(m_t\colon V\to[0,1])_{t\geq 0}$, which we call the \emph{fragmentation process originating at $o$}. It follows the same Markovian transition law as $(f_t)_{t\geq 0}$, but its initial value is the indicator vector of $o$. That is, $m_0(o)=1$, and $m_0(v)=0$ for $v\in V\setminus\{o\}$. It is well known (see  \cite {gantert2024averaging}, [Eq.~6 and Corollary 2.3] or \cite{MR2908618}) that for every $t>0$, we can embed the opinion dynamics $(f_s)_{s\in[0,t]}$ and the fragmentation process $(m_s)_{s\in[0,t]}$  in the same probability space, so that $f_0$ and $m_t$ are independent, and 
\begin{equation}\label{eq:duality}
  f_t(o)=\sum_{v\in V}m_t(v)f_0(v)\quad\text{a.s.}  
\end{equation}
The idea  behind this embedding is that for every $s\in[0,t]$, the opinion  $f_t(o)$ can be expressed as a weighted average of $\{f_{t-s}(v)\}_{v\in V}$, with weights   $\{m_{s-}(v)\}_{v \in V}$. This can be verified by induction on the clock rings in the support of the fragmentation process. For example, if $\{o,v\}$ is the last edge to ring before time $t$, then  $f_t(o)=\big(f_{\tau_1}(o)+f_{\tau_1}(v)\big)/2$  for some   $\tau_1<t$. Similarly, $f_{\tau_1}(o)$ and $f_{\tau_1}(v)$ are averages of opinions at earlier times, etc. 

The representation \eqref{eq:duality} provides a powerful tool
for analyzing the edge-averaging process: Since $f_t(o)$ is expressed as a weighted sum of i.i.d.\ random variables, 
classical concentration inequalities (e.g., Hoeffding’s inequality) can be applied to bound deviations of $f_t(o)$ from its mean. The more the measure $m_t(\cdot)$ is dispersed, the more the opinion $f_t(o)$ is concentrated around its mean.

We measure the dispersion of the fragmentation process via the quadratic sum 
\begin{equation}
    Q_t:=\sum _{v\in V} m_t(v)^2 \,.
\end{equation}
Our goal is to bound $Q_t$ from above.
For an edge $e=\{u,v\}\in E$, let $\delta _e:=| m_t(u)-m_t(v)|$. When an edge $e\in E$ rings, the value of $Q_t$ decreases by $\delta_e^2/2$. Thus, $Q_t$ is weakly decreasing and its drift is
\begin{equation}
    D(Q_t):= \lim_{h\downarrow 0}\tfrac 1 h\EE \big[ Q_{t+h}-Q_t\mid \mathcal F _t \big]= -\sum _{e\in E} \frac{\delta_e^2}{2}=  -\frac{\mathcal E_t}{2} \,,
\end{equation}
where 
\begin{equation}
      \mathcal E_t:= \sum_{\{u,v\} \in E} (m_t(u)-m_t(v))^2 \,.
\end{equation}
Hence, in order to bound $Q_t$ from above, it will be useful to bound $\mathcal E_t$ from below, as we do in the following lemma. 
\begin{lem}\label{lem:energy}
    Let $t>0$. Suppose that $Q_t\ge 2/n$, where $n=|V|$ (This inequality holds automatically if $V$ is infinite.) Then $\mathcal E_t\ge Q_t^3/8$.
\end{lem}

\begin{proof}
Let $M:=\max _{u\in V}m_t(u)$ and let $u_0$ be a vertex achieving the maximum. We claim that there is a vertex $w$ for which $m_t(w) \le M/2$. This clearly holds when $V$ is infinite.  For finite $V$, our assumption gives 
\begin{equation}\label{eq:QM}
\frac2n \le Q_t \le M\sum _{u\in V} m_t(u)=M\,.    
\end{equation}
The claim follows, since $\min_w m_t(w) \le \frac1n \sum _{u\in V} m_t(u) =\frac1n$. 
Let $\gamma=(u_0,\ldots ,u_{\ell})$ be a shortest path from $u_0$ to the set $\{w \in V: m_t(w) \le M/2\}$.  
By the Cauchy-Schwarz inequality, 
\begin{equation}
\begin{split}
    \big( m_t(u_0)-m_t(u_{\ell}) \big)^2= \Big( \sum_{i=1}^{\ell}  m_t(u_i)-m_t(u_{i-1}) \Big)^2\\
    \le \ell \sum _{i=1}^\ell  \bigl(m_t(u_{i})-m_t(u_{i-1})\bigr)^2 \le \ell \mathcal E_t.
\end{split}
\end{equation}
Recall from \eqref{eq:QM} that $Q_t \le M$, and  our choice  of $u_0$ and $u_\ell$  gives $m_t(u_0)-m_t(u_{\ell})\ge M-M/2=M/2$.  Therefore,
\begin{equation}
\label{eq:bound}
    Q_t^3 \le M^3\le 4M \big( m_t(u_0)-m_t(u_{\ell}) \big) ^2\le 4 M \ell \mathcal E_t.
\end{equation} 

The minimality of the path $\gamma$ implies that $m_t(u_i)\geq M/2$ for all $i=0,\ldots,\ell-1$ and therefore $\ell M \le 2\sum _{i=0}^{\ell-1}  m_t(u_i)\le 2$. Substituting this estimate into \eqref{eq:bound} finishes the proof of the lemma.
\end{proof}

\begin{lem}\label{lem:Qtail}
    For all $t>0$, we have  
\begin{equation}
    \mathbb P (Q_t\ge 6t_*^{-1/2}) \le  \exp (-t_*/30) \,,   
\end{equation}
where $t_*=\min(t,|V|^2)$. 
Consequently, $\EE[Q_t]\leq 8t_{*}^{-1/2}$.
\end{lem}

\begin{proof}
    Let $\mathcal F_t$ be the $\sigma $-algebra generated by the fragmentation process up to time $t$, and define the process $$Y_t:=\mathds 1 \{ Q_t\ge 2/n \} \exp (-1/Q_t^2) \,.$$
    The heart of the proof is showing that the drift of $Y_t$ satisfies
    \begin{equation}\label{eq D(Y_t)}
        D(Y_t):=\lim_{h\downarrow 0}\tfrac 1 h\EE \big[ Y_{t+h}-Y_t\mid \mathcal F _t \big] \leq -Y_t/{16}.
    \end{equation}
    Inequality \eqref{eq D(Y_t)} trivially holds when $Y_t=0$ and therefore from now on we work on the event that $Q_t\ge 2/n$. 
     Recall that when an edge $e=\{u,v\}\in E$ rings, the value of $Q_t$ decreases by $\delta_e^2/2$, where $\delta _e=| m_t(u)-m_t(v)|$. Thus, 
\begin{equation}
    -D(Y_t)\ge \sum _{e\in E} \big[\exp (-Q_t^{-2})-\exp (-(Q_t-\delta _e^2/2)^{-2})\big].
\end{equation}
In order to estimate the last sum, we write $\delta_{\mathrm{max}}:=\max_{e\in E}\delta_e$ and consider two cases.

\noindent{\bf Case a}: $\delta^2_{\mathrm{max}}\leq Q_t^3$. \newline In this case,      Claim~\ref{claim:calc} with $\lambda =\delta _e^2/(2Q_t^3)\le 1/2$ gives
\begin{equation}
    -D(Y_t)\ge \exp (-Q_t^{-2})\sum _{e\in E} \frac{\delta _e^2}{2Q_t^3}= \frac{Y_t\mathcal E_t}{2Q_t^3} \ge \frac{Y_t}{16},
\end{equation}
where in the last inequality we use Lemma~\ref{lem:energy}.  

\noindent{\bf Case b}: $\delta^2_{\mathrm{max}}> Q_t^3$. \newline
In this case, we  fix an edge $e_*\in E$ for which $\delta _{e_*}^2\ge Q_t^3$. Then 
\begin{equation}
\begin{split}
     -D(Y_t)&\ge \exp\bigl(-Q_t^{-2}\bigr)-\exp\bigl(-(Q_t-\delta _{e_*}^2/2)^{-2}\bigr) \\
     &\ge \exp\big(-Q_t^{-2}\big)-\exp \big(-(Q_t-Q_t^3/2) ^{-2}\big) \\
     &\ge  \exp\big(-Q_t^{-2}\big)/2=Y_t/2,
\end{split}
\end{equation}
where in the second inequality we used Claim~\ref{claim:calc} once again (this time with $\lambda =1/2$). This concludes the proof of \eqref{eq D(Y_t)}.

Define $Z_t:=\exp(t/16)Y_t$. Since $$D(Z_t)=Z_t/16+\exp(t/16)D(Y_t) \,,$$ the inequality \eqref{eq D(Y_t)} implies that $D(Z_t)\leq 0$, so  $\{Z_t\}$ is a supermartingale. Thus, using that $\mathbb E [Z_0]=\mathbb E [Y_0]=1/e$ we obtain
$$\EE[Y_t]=\exp(-t/16)\EE[Z_t]\leq \exp(-t/16)\EE[Z_0] \leq\exp(-t/16).$$ Hence\footnote{The inequality $Y_t \ge e^{-t/36}$ holds if and only if  $Q_t \ge 2/n$ and $Q_t \ge 6t^{-1/2} $.}, by Markov's inequality, we have
\begin{equation}
    \mathbb P \big( Q_t\ge \max (6t^{-1/2},2/n) \big)=\mathbb P \big( Y_t \ge e^{-t/36} \big) \le e^{-t/30}.
\end{equation}
Since $t_*^{-1/2}=\max(t^{-1/2},n)$, the first statement of the lemma follows from the last bound.

The second statement, concerning the expectation of $Q_t$, is clear if $t_* \le 64$,  since $Q_t\le 1$ always.  On the other hand,  if $t_* \ge 64$, then $ \exp(t_*/30)>\frac12 t_*^{1/2}$ by comparing derivatives.
Therefore, the first statement of the lemma yields
\begin{equation}
    \mathbb E [Q_t] \le 6t_*^{-1/2}+ \mathbb P (Q_t\ge 6t_*^{-1/2})<8t_*^{-1/2} \,.
\end{equation}
\end{proof}

\begin{claim}\label{claim:calc}
For any $0< Q\le 1 $ and $0\le \lambda \le 1/2$, we have that 
\begin{equation}
    \exp (-(Q-\lambda Q^3)^{-2}) \le (1-\lambda )\exp (-Q^{-2}).
\end{equation}
\end{claim}

\begin{proof}
The inequality $1/(1-x)^2\ge 1+2x$ for $x< 1$ yields
\begin{equation}
    (Q-\lambda Q^3)^{-2} = Q^{-2}(1-\lambda Q^2)^{-2}\ge  Q^{-2}+2\lambda.
\end{equation}
Thus, using the inequality $\exp(-2x)\leq 1-x$ for $0\leq x\leq 1/2$, we have
\begin{equation}
    e^{-(Q-\lambda Q^3)^{-2}} \le e^{-2\lambda}e^{-Q^{-2}}\le (1-\lambda )e^{-Q^{-2}}.
\end{equation}
This concludes the proof of Claim~\ref{claim:calc}.
\end{proof}

\section{Proofs: bounded initial opinions}
The proof of Theorem \ref{theorem finite} proceeds in two steps. First, we establish Theorem \ref{theorem local bounded}, which uses the i.i.d.\
nature of the initial opinions and the  dispersion of the  fragmentation process, as measured by the quadratic sum $Q_t$. Applying a concentration inequality, we deduce that with high probability, the opinions of all
nodes lie within $\eps$ of consensus by time $t_0$ of order $(\log^2 n)/\eps^4$. 
Second, on the unlikely event that not all opinions are near the mean
by time $t_0$, we switch to Proposition \ref{prop:arbitraryf_0}, which only requires that the initial opinions lie in $[-1,1]$, with no assumption of independence. It shows that  
starting from an arbitrary configuration, the probability of remaining outside $\eps$-consensus at time $t$ decays exponentially at rate $O(n^{-2})$. By combining these
two bounds, we see that the expected time to consensus is dominated by $t_0$, plus a negligible additional term.  This yields the final $O\big((\log^2 n)/\eps^4\big)$ estimate on the expected consensus time.

\begin{proof}[Proof of Theorem~\ref{theorem local bounded}]
    We may assume that $t_*\ge 144$, as otherwise the bound holds trivially. Consider the fragmentation process $\{m_t\}$ originating at $o$ and recall that $Q_t:=\sum _{u\in V} m_t(u)^2$.
By Lemma~\ref{lem:Qtail}, 
\begin{equation}\label{eq:Q}
\mathbb P(Q_t\geq 6t_*^{-1/2})<\exp({-t_*/30}).    
\end{equation}

Next, let $\mathcal F _\VarClock$ be the $\sigma $-algebra generated by all the clock rings. Since the initial opinions are  i.i.d.\ random variables taking values in $[-1,1]$, the identity \eqref{eq:duality} and Hoeffding's Inequality implies that for $\eps \le 1$,
\[
\mathbb P\big(|f_t(o)-\mu |>\eps\mid \mathcal F_\VarClock\big)\le 2\exp \Big( -\frac{\eps^2}{2 Q_t} \Big).
\]
On the complement of the event in \eqref{eq:Q}, the last expression is bounded by $2 \exp ( -\eps^2 \sqrt{t_*}/12 ) $;   therefore, we obtain 
\begin{equation}
\begin{split}
    \mathbb P \big( |f_t(o)-\mu |\ge \eps \big) &=\mathbb E \big[ \mathbb P(|f_t(o)-\mu |>\eps\mid \mathcal F _\VarClock) \big] \\ 
    &\le \mathbb P  (Q_t\ge 6t_*^{-1/2})+2\exp \Big( -\frac{\eps^2 \sqrt{t_*} }{12  } \Big) \\
    &\le 3\exp \Big( -\frac{\eps^2 \sqrt{t_*} }{12  } \Big) ,
\end{split}
\end{equation}
where in the last inequality we used \eqref{eq:Q}.
\end{proof}
Next, we prove the general bound on  consensus time that holds for arbitrary initial opinions in $[-1,1]^V$; such a bound could be deduced from the results of \cite{boyd2006randomized}, but we include a direct argument for the reader's convenience.
\begin{prop}\label{prop:arbitraryf_0}
 Suppose $n=|V|<\infty $. For every function $f_0:V\to [-1,1]$ and every $t>0$, we have $$\mathbb P (\tau_{\eps } > t) \le \frac{2n}{\eps^2}\exp(-2t/n^2) \,.$$
\end{prop}

\begin{proof}
    Let $\bar \mu :=\frac{1}{n}\sum_{u\in V}f_0(u)$ be the average of the initial opinions (which is also the average opinion throughout the process). The sum  $H_t:= \sum_{u\in V} (f_t(u)-\bar \mu ) ^2 $ never increases. 
    Moreover, if the edge $e=\{u,v\}$ is updated at time $t$, then   $H_t-H_{t-}=-\frac{1}{2}|\nabla f_{t-}(e)|^2$, where
    $|\nabla f(e)|:=|f(u)-f(v)|$.
    Therefore, the drift (see \eqref{eq D(Y_t)}) of $H_t$ equals
    \begin{equation} \label{drifth}
    D(H_t) =-\frac{1}{2}\sum_{e \in E} |\nabla f_t(e)|^2 \,.  
\end{equation}
Suppose that $\osc(f_t)=f_t(w_+)-f_t(w_-)$ and $(w_0,\dots,w_\ell)$ is a shortest path in $G$ from $w_0=w_-$ to $w_\ell=w_+$.
Writing $e_j:=\{w_{j-1},w_j\}$,  we have
$$    
\osc(f_t) = \sum_{j=1}^\ell \big(f_t(w_j)-f_t(w_{j-1})\big)  \le    \sum_{j=1}^\ell |\nabla f_t(e_j)| \,.
$$ 
The  Cauchy-Schwarz inequality and \eqref{drifth} yield 
\begin{equation}
\label{cauchy1}   
\osc(f_t)^2  \le \ell  \sum_{j=1}^\ell |\nabla f_t(e_j)|^2 
 \le -2n D(H_t)\,.
\end{equation}

 Observe that the quadratic polynomial $h_t(x):= \sum_{u\in V} (f_t(u)-x ) ^2$ is minimized at $x=\bar\mu$. In particular $$H_t=  h_t(\bar \mu) \le h_t\Big(\frac{\max f_t+\min f_t}{2}\Big) \le \frac{n}{4}\osc(f_t)^2 \,.
 $$
 Combining this with \eqref{cauchy1}, we deduce that for all $t \ge 0$,
$$H_t  \le -\frac{n^2}{2} D(H_t) \,.
$$
Let $S_t:=\exp(2t/n^2) H_t$. By the derivation preceding Lemma 2.2 in~\cite{elboim2024asynchronous}, we may apply the product rule to calculate   
$$D(S_t)=\tfrac{2}{n^2}\exp(2t/n^2) H_t+\exp(2t/n^2)D(H_t) \le 0\,,$$
so $\{S_t\}$ is a supermartingale. Therefore,
\begin{equation}
    \begin{split}
\PP(\tau_\eps>t)&=\PP(\osc (f_t)>\eps) \le \PP (H_t>\eps^2/2)\\
&=\PP\big(S_t> \exp(2t/n^2)\frac{\eps^2}{2}\big) \\
& \le \EE[S_0] \exp(-2t/n^2)\frac{2}{\eps^2} \le \frac{2n}{\eps^2}\exp(-2t/n^2)\,.\qedhere
    \end{split}
\end{equation}
\end{proof}

\begin{proof}[Proof of Theorem~\ref{theorem finite}]
    Let $t_0:=10^5 \log ^2 n/\eps^4$ and $t_1:=n^3$. By Theorem~\ref{theorem local bounded}, we have that 
    \begin{equation}
        \mathbb P (\tau _\eps \ge t_0 ) \le \sum _{u\in V} \mathbb P \big( |f_{t_0}(u)-\mu |\ge \eps /2  \big) \le  3n e^{-6\log n} =  3n^{-5}.
    \end{equation}
Thus, by Proposition~\ref{prop:arbitraryf_0},
\begin{equation}
\begin{split}
    \mathbb E [\tau _\eps ]  &\le t_0+\int\nolimits _{t_0}^{t_1}  \mathbb P (\tau _\eps >t) \, dt + \int\nolimits _{t_1}^{\infty }  \mathbb P (\tau _\eps >t)\, dt \\
    &\le t_0 + 3(t_1-t_0)n^{-5} + \int _{t_1}^\infty \frac{2n}{\eps^2}\exp(-2t/n^2)\, dt   \\
    &\le t_0 + 3n^{-2}  + \eps^{-2} n^3e^{-2n} \le 2t_0. \qedhere
\end{split}
\end{equation}
\end{proof}

\section{Proofs: initial opinions in $L^p$}

We start with the following special cases of Theorem~\ref{theorem:Lp}.
\begin{lem}\label{lem:L1L2}
Fix a vertex $o\in V$. Then, for every $t>0$,
$$\|f_t(o)\|_1\leq \|f_0(o)\|_1\, .$$
If $\mathrm{Var}(f_0(o))<\infty$, then we also have
$$\mathrm{Var}(f_t(o))=\EE[Q_t]\mathrm{Var}(f_0(o)).$$
\end{lem}
\begin{proof}
Fix $o\in V$ and $t>0$, and let $m_t(\cdot)$ be the fragmentation process originating at $o$. Recall that $\mathcal F_\VarClock$ is the $\sigma $-algebra generated by the clock rings, $m_t(\cdot)$ is $\mathcal F _\VarClock$-measurable, the random variables $\{f_0(v)\}_{v\in V}$ are i.i.d.\ independent of $\mathcal F _\VarClock$, and $f_t(o)=\sum_{v\in V} m_t(v)f_0(v)$. 

The first part of the lemma follows by observing that
\begin{equation}
    \begin{split}
        \EE\big[ |f_t(o)|\ \big| \ \mathcal F _\VarClock \big]
        &\leq\EE\Big[\sum_{v\in V}m_t(v)|f_0(v)| \ \big| \ \mathcal F _\VarClock\Big]\\
        &=\sum_{v\in V}m_t(v)\EE[|f_0(v)|]=\|f_0(o)\|_1.
    \end{split}
\end{equation}

For the second part of the lemma, we write
\begin{equation}
    \begin{split}
        \mathrm{Var}\big( f_t(o) \ \big| \  \mathcal F _\VarClock \big)&=\mathrm{Var}\Big(\sum_{v\in V}m_t(v)f_0(v) \ \big| \ \mathcal F _\VarClock\Big)\\
        &=\sum_{v\in V}m_t^2(v)\mathrm{Var}(f_0(v))=Q_t\mathrm{Var}(f_0(o)).
    \end{split}
\end{equation}
Taking expectations of both sides conclude the proof.
\end{proof}

\begin{proof}[Proof of Theorem~\ref{theorem:Lp}]
For $p=1$ and $p=2$, the inequality \eqref{eq:p}   follows from Lemmas~\ref{lem:Qtail} and \ref{lem:L1L2}. 

We now consider the case $p>2$, and assume w.l.o.g.\ that $\mu=0$. Note that conditioned on $\mathcal F _\VarClock$, the random variable $f_t(o)$ is  a sum of independent mean zero random variables. Thus, by the Marcinkiewicz–Zygmund inequality (see, e.g., \cite[Section~10.3]{chow2012probability}) with the optimal constant due to Burkholder~\cite[Theorem 3.1]{burkholder1988sharp}, we have
    \begin{equation}
    \begin{split}
        \mathbb E \big[ |f_t(o)|^p \mid \mathcal F _\VarClock \big]&=\mathbb E \Big[ \Big| \sum _{v\in V}m_t(v)f_0(v) \Big|^p \mid \mathcal F _\VarClock \Big] \\
        &\le (p-1)^p \mathbb E \Big[ \Big| \sum _{v\in V}m_t^2(v)f_0^2(v) \Big| ^{p/2} \mid \mathcal F _\VarClock \Big]\\
        &=(p-1)^p Q_t^{p/2} \mathbb E \Big[ \Big| \sum _{v\in V}\frac{m_t^2(v)}{Q_t}f_0^2(v) \Big| ^{p/2} \mid \mathcal F _\VarClock \Big].
         \end{split}
    \end{equation}
    Thus, using the convexity of the function $x\mapsto x^{p/2}$ and Jensen's inequality we obtain
    \begin{equation}
    \begin{split}
        \mathbb E \big[ |f_t(o)|^p \mid \mathcal F _\VarClock \big]&\le (p-1)^p Q_t^{p/2} \mathbb E \Big[  \sum _{v\in V}\frac{m_t^2(v)}{Q_t}|f_0(v)|^{p}  \mid \mathcal F _\VarClock \Big] \\
        &=(p-1)^p Q_t^{p/2}   \sum _{v\in V}\frac{m_t^2(v)}{Q_t}\mathbb E \big[ |f_0(v)|^{p}  \mid \mathcal F _\VarClock \big]\\
        &= (p-1)^pQ_t^{p/2}\mathbb E\big[ |f_0(o)|^p\big].
    \end{split}
    \end{equation}
    By taking expectation and then raising to the power $1/p$ on both sides, we get
    \begin{equation}\label{eq:356789}
    \|f_t(o)\|_p \leq (p-1)\|f_0(o)\|_p\|Q_t^{1/2}\|_p.    
    \end{equation}
    Next, using that $Q_t\le 1$ and Lemma~\ref{lem:Qtail} we obtain 
    \begin{equation}
        \mathbb E [Q_t^{p/2}] \le (6t_*^{-1/2})^{p/2} +\mathbb P (Q_t \ge 6t_*^{-1/2}) =O(t_*^{-p/4})
    \end{equation}
    and therefore $\|Q_t^{1/2}\|_p=O(t_*^{-1/4})$. Substituting this estimate into \eqref{eq:356789} completes the proof of the theorem when $p> 2$.

    Next, suppose that $p\in (1,2)$. Let $(\Omega_0,\mathcal F_0,\mathbb P_0)$ be a probability space on which the initial opinion at a single vertex is defined\footnote{One can take $\Omega_0=\mathbb R$ equipped with the Borel $\sigma$-field and let $P_0$ be the law of $f_0(o)$.}. Let $(\Omega_\VarClock,\mathcal F_\VarClock,\mathbb P_\VarClock)$ be a probability space on which the Poisson clocks, and hence the fragmentation process $m_t(\cdot)$   originating at $o$, are defined. Then the edge-averaging process is defined on $(\Omega,\mathcal G ,\mathbb P):= (\Omega_0,\mathcal F_0,\mathbb P_0)^V \otimes (\Omega_\VarClock,\mathcal F_\VarClock,\mathbb P_\VarClock)$. A state $\omega\in \Omega$ has the form $\omega=((\omega_v)_{v\in V},\omega_\VarClock)$. The initial opinion of every vertex $v\in V$ is a function of $\omega_v$ and the fragmentation process is a function of $\omega_\VarClock$. Define the linear operator $T_t:L^1(\Omega_0,\mathcal F_0,\mathbb P_0)\to L^1(\Omega,\mathcal G ,\mathbb P)$ as follows: 
    \begin{equation}
    \label{eq:def_T}
      (T_t\,F)(\omega):=\sum_{v\in V}m_t(v)(\omega_ \VarClock )F(\omega_v)-\EE_0[F],  
    \end{equation}
    where $\EE_0$ denotes the expectation operator on $(\Omega_0,\mathcal F_0,\PP_0)$.
    The definition of $T_t$ implies that $T_t\,F$ is $f_t(o)-\mu$, where $f_0(v):=F(\omega_v)$  for all $v\in V$. By the Riesz-Thorin theorem (see, e.g., \cite[Page 52]{stein2011functional} or Wikipedia), we have that 
    \begin{equation}\label{eq:thorin}
     \|T_t\|_{L^p(\PP_0  )\to L^p(\PP)} \le \|T_t\|_{L^1(\PP_0)\to L^1(\PP)} ^{2/p-1}\|T_t\|_{L^2(\PP_0 )\to L^2(\PP)}^{2-2/p}. 
    \end{equation}
    Lemma~\ref{lem:L1L2} implies that $\|T_t\|_{L^1(\PP_0 )\to L^1(\PP)} \le 2$, and together with Lemma~\ref{lem:Qtail}, that $\|T_t\|_{L^2(\PP_0 )\to L^2(\PP)} \le 3t_*^{-1/4}$. Thus, by setting $C_p:= 2^{2/p-1}3^{2-2/p}$, we get
    \[
    \|f_t(o)\|_p \leq C_p  t_*^{\frac{1-p}{2p}}\|f_0(o)\|_p,
    \]
    as claimed. 

    Next, suppose that $G$ is infinite, $p \in [1,2)$, and $f_0(v)=F(\omega_v)$, where $F \in L^p(\Omega_0)$. Given $\eps>0$, there exists 
      $H$ in $L^2(\Omega_0)$ such that $\|F-H\|_{L^p(\Omega_0)} < \eps$. 
Therefore,
\begin{equation}
    \begin{split}
        \|f_t(o)-\mu\|_p&=\|T_t(F)\|_p\le \|T_t(H)\|_2+\|T_t(F-H)\|_p \\
        &\le C_2 \|H\|_2 \cdot t^{-1/4}+C_p\eps t^{\frac{1-p}{2p}} \le 2C_p\eps t^{\frac{1-p}{2p}}\,,
    \end{split}
\end{equation}
    for $t$ sufficiently large. This completes the proof.
\end{proof}

\begin{proof}[Proof of Theorem~\ref{theorem infinite}]
Recall that  $f_0(o)\in L^p$ for some $p>4$. Fix $\eps>0$. By Theorem~\ref{theorem:Lp} and Markov's inequality,
\begin{equation}\label{fkbound}
\PP(|f_t(o)-\mu| > \eps)\leq \frac{\|f_t(o)-\mu\|_p^p}{\eps^p}\leq C t^{-p/4}, 
\end{equation}
where the constant $C$ depends on $\eps ,p$ and $\|f_0\|_p$. For every $k \ge 1$, define the event  $$A_k:= \{ \exists t \in [k-1,k] : \, | f_t(o)-\mu|>\eps\} \,. $$
Then using the strong Markov property at time
$$\tau_k=\inf\{t>k-1:|f_t(o)-\mu|>\eps\} \, $$
and considering the conditional probability, given $A_k$,  that no clock adjacent to $o$ rings during  $(\tau_k,k]$, we infer that
\begin{equation} \PP(| f_k(o)-\mu|>\eps) \ge e^{-\mathrm{deg}(o)} \PP(A_k). 
\end{equation}

Thus by  \eqref{fkbound} we have $\mathbb P(A_k)\le C e^{{\rm deg} (o)} k^{-p/4}$. This bound is summable in $k$, so by Borel-Cantelli, almost surely only finitely many of the events $A_k$ occur. Since $\eps>0$ is arbitrary, the proof that $f_t(o)$ converges to $\mu$ almost surely is complete.
\end{proof}

\section{Lower bounds}\label{sec:lowerbound}

In this section, we prove the tightness of the upper bounds in Theorems~\ref{theorem finite}, \ref{theorem local bounded}, and \ref{theorem:Lp}.  The setting is an $n$-cycle with i.i.d.\ unbiased $\pm 1$ initial opinions. To prove the global lower bound (Theorem~\ref{theorem tightness}),  we split the cycle into $n/\ell$ arcs of length $\ell:=C (\log n)/ \eps^{2}$. For the $i$'th
arc, we define  a  biased  measure $\nu_i$ wherein the vertices in that arc have a
higher probability (e.g., $(1+3\eps )/2$) of starting with opinion $+1$.
In each arc, starting with $\nu_i$, local biases persist long enough to keep
 node opinions elevated (Lemma~\ref{lem:f_t}).
We deduce that starting with the average measure $\nu_*$, it is likely that at time $t_0=c\ell^2$,  one or more of the $n/\ell$ arcs will remain sufficiently biased. This forces the entire system’s opinion range to
exceed $2\eps$ with non-negligible probability, making $\mathbb E [\tau _\eps ]$ at least on the order of $(\log ^2 n)/\eps ^4$.  A total-variation argument ensures that the average
$\nu_*$ of the $\nu_i$ is close to the original uniform  $\pm 1$ measure, so events that are likely under $\nu_*$, must also occur with decent probability under
the original uniform measure.

For the local lower bound (Proposition~\ref{prop:lowlocal}), we focus  on a single arc around $o$,  and show that $\mathbb P (f_t(o)\ge \eps )$ is at least of order $\exp (-c\eps ^2\sqrt{t})$ at time $t$. This matches the upper bound from Theorem~\ref{theorem local bounded}, demonstrating its tightness.  The same construction yields tightness of Theorem~\ref{theorem:Lp} when $p \ge 2$.
 Finally, for the case $p \in [1,2)$ of Theorem \ref{theorem:Lp}, initial opinions $\pm 1$ no longer suffice, and we prove that symmetric stable initial opinions yield a lower bound on consensus time that almost matches the upper bound.

\begin{thm}\label{theorem tightness}
    Consider the edge-averaging process on the $n$-cycle with i.i.d.\ $\mathrm{Uniform}\{-1,1\}$ initial opinions, and let  $\eps \in (0,1/3)$. Then provided $n>n_\eps$, we have $$\EE[\tau_\eps]\geq c \log^2(n)/\eps^4,$$
    where $c>0$ is a universal constant.
\end{thm}

Recall that $(m_t)_{t\ge 0}$ denotes the fragmentation process originating at $o$ and that $Q_t=\sum _{v\in V}m_t^2(v)$.

\begin{lem}\label{lem:omega}
Let $G=(V,E)$ be either the $n$-cycle or the infinite path $\mathbb Z$ and let $o\in V$. Suppose that $t_*=\min(t,n^2)\ge 120$ and let $J$ be the set of vertices in $V$ whose distance from $o$ is less than $4\sqrt{t}$. Define the event 
\begin{equation}
    \Omega := \big\{ m_t(J) \ge 1/2,  \  Q_t \le 6t_*^{-1/2} \big\},
\end{equation}
where $m_t(J):=\sum _{v\in J} m_t(v)$. Then $\mathbb P (\Omega ) \ge 4/5$.
\end{lem}

\begin{proof}
The second inequality in the definition of $\Omega $ holds with probability at least $1-e^{-4}$ by Lemma~\ref{lem:Qtail}, using that $t_*\ge 120$. 
Next, note that $\mathbb E [m_t(u)]=\mathbb P _o(X_t=u)$, where $(X_s)_{s \ge 0}$ is a simple  continuous-time random walk starting from $o$, governed by Poisson clocks on the edges: when the clock of an edge incident to the particle rings, the particle crosses it with probability $1/2$.
Thus, by Chebyshev's inequality, 
\begin{equation}
     \mathbb E [m_t(J^c)] =\mathbb P \big( d(X_t,o) \ge 4\sqrt{t} \big)\le \tfrac{1}{16} \,,
\end{equation}
so $\mathbb P (m_t(J^c) \ge \tfrac12 ) \le \tfrac18$. Therefore, $\PP(\Omega^c) \le \tfrac18+e^{-4}<\tfrac15 $.
\end{proof}

\begin{lem}\label{lem:f_t}
    Let $o$ and $J$ be as in Lemma~\ref{lem:omega} and let $\eps \in (0,1/3)$. Suppose that the initial opinions $\{f_0(v)\}$ are independent $\pm 1$ valued random variables with distribution $\nu _J$, where  $\nu _J (f_0(v)=1)=1/2$ for $v\notin J$ and $\nu _J (f_0(v)=1)=(1+3\eps)/2 $ for $v\in J$. Let $\mathbb P _J:=\nu _J \times \mathbb P _\VarClock$. If $t_* \ge (10/\eps)^4$, then 
    \begin{equation}
        \mathbb P_J \big( f_t(o)\ge \eps  \big) \ge 3/5.
    \end{equation}
\end{lem}
\begin{proof}
Note that 
$$\var_J(f_t(o) \mid \mathcal F 
 _\VarClock) \le Q_t $$
 and 
$$\EE_J[f_t(o) \mid \mathcal F 
 _\VarClock] = 3\eps m_t(J) \ge \frac{3\eps}{2}\,\mathds 1_\Omega \,, $$
 so Chebyshev's inequality and the assumption on $t_*$ yield
 $$  
        \mathbb P_J \big( f_t(o)< \eps \mid \mathcal F _\VarClock \big) \,\mathds 1_\Omega \le \frac{4Q_t}{\eps^2} \,\mathds 1_\Omega \le  \frac{24}{\eps^2\sqrt{t_*}} \le \frac14\,.
  $$
  Thus by Lemma \ref{lem:omega},  
  $$\mathbb P_J \big( f_t(o) \ge \eps \big)  \ge \EE\big[\PP_J \big( f_t(o)\ge \eps \mid \mathcal F _\VarClock \big) \,\mathds 1_\Omega \big]  \ge \frac34  \cdot \frac45 \,. \qedhere$$
\end{proof}    
\begin{proof}[Proof of Theorem~\ref{theorem tightness}]
    Let $\ell := \lfloor(\log n)/(36\eps ^2)\rfloor$ and let $k:=\lfloor n/(2\ell) \rfloor $. Let $w_1,\dots ,w_k$ be $k$ vertices with distance  at least $2\ell $ from each other. For $i\le k$, let $J_i$ be the arc centered at $w_i$ consisting of $2\ell-1$ vertices,  and define $\psi_i :\{-1,1\}^V\to \mathbb R$ by
    \begin{equation}
     \psi_i(x):=\prod _{v\in J_i} (1+3\eps x_v).   
    \end{equation}
    Let $\nu $ be the uniform measure on $\{-1,1\}^V$ and define $\nu _i$ by $d\nu _i= \psi _i d\nu $, so for $t=\ell^2/16$, the measure $\nu_i$ is a rotation of $\nu_J$ from Lemma  \ref{lem:f_t}. Finally, let $\nu _*:=\frac{1}{k} \sum _{i=1}^k \nu _i$.    Our next goal is to show that $\nu$ and $\nu_*$ are close in total variation. We first compute the $L^2$ norms of $\psi_i$:
    \begin{equation} \label{normpsi_i}
    \begin{split}
\EE_\nu[\psi_i^2]&=
\prod _{v\in J_i} \frac{(1+3\eps)^2+(1-3\eps)^2}{2} \\
&\le(1+9\eps^2)^{2\ell} 
\le \exp (18\eps^2 \ell ) \le \sqrt{n}\,,
  \end{split}
    \end{equation}
    by our choice of $\ell$.  Thus, since $\psi _i$ are independent with respect to $\nu $, their average $\psi_*:=\tfrac1k\sum_{i=1}^k \psi_i$ satisfies
\begin{equation} \label{normpsi_*}
\var_\nu(\psi_*)=\tfrac{1}{k^2} \sum_{i=1}^k \var_\nu(\psi_i) \le \frac{\sqrt{n}}{k} \le n^{-1/3} \,,
\end{equation}
provided $n$ is large enough.
Therefore,
\begin{equation} \label{totvar}
2\|\nu-\nu_*\|_{\rm TV}=
\|\psi_*-1\|_{L^1(\nu)} \le 
\|\psi_*-1\|_{L^2(\nu)} \le n^{-1/6}\,,
\end{equation}
for large $n$.  Next, set $t_0:=\ell^2/16$ and 
consider the event
$$A^+:=\{\exists w \in V: f_{t_0}(w) \ge \eps \}\,.
$$
For $i \le k$,
define $\PP_i=\nu_i \times \PP_{\VarClock}$. Then, for sufficiently large $n$, we may use Lemma \ref{lem:f_t} with $t=t_0$ to obtain
$$\PP_i(A^+)\ge \PP_i(   f_t(w_i) \ge \eps ) \ge 3/5 \,.$$
 Therefore, the measure
$$\PP_*=\tfrac1k\sum_{i=1}^k \PP_i =\nu_* \times \PP_{\VarClock}$$
also satisfies
$\PP_*(A^+) \ge 3/5.$
Since $$\|\PP_*-\PP\|_{\rm TV}=\|\nu_*-\nu\|_{\rm TV} \le n^{-1/6} \,,$$
we deduce that 
$\PP(A^+) \ge 4/7.$
 Similarly, $\PP(A^-) \ge 4/7$, where
  $A^-:=\{\exists w \in V: f_{t_0}(w) \le -\eps \}$.
Consequently,
$$\PP\big(\osc(f_{t_0}) \ge 2\eps\big) \ge \PP(A^+\cap A^-) \ge 1/7\,. $$
Finally, we obtain
$$\EE[\tau_\eps] \ge t_0\PP(\tau_\eps \ge t_0) \ge \frac{t_0}{7}=\frac{\ell^2}{112} \ge c \frac{\log^2 n}{\eps^4} \,.\qedhere$$
\end{proof}

The next proposition gives a lower bound on $\mathbb P (f_t(o)\ge \eps )$ that matches the upper bound in Theorem~\ref{theorem local bounded}, up to the constant in the exponent.

\begin{prop}\label{prop:lowlocal}
    Consider the edge-averaging process on $G=(V,E)$ where $G$ is either the $n$-cycle or the infinite path $\mathbb Z$ with i.i.d.\ $\mathrm{Uniform}\{-1,1\}$ initial opinions, and let  $\eps \in (0,1/3)$. If $t^*=\min(t,n^2)\ge (10/\eps )^4$, then 
   \begin{equation}
    \mathbb P \big( f_t(o) \geq \eps  \big) \ge \tfrac{1}{80}\exp( -72\eps ^2 \sqrt{t_*})\,.
\end{equation}
\end{prop}
\begin{proof}
Let $\nu$ be the uniform distribution on $\{1,-1\}^V$ and $\PP:=\nu\times\PP_\VarClock$. Let $J$ be as  in Lemma~\ref{lem:omega}, so that $|J| \le 8\sqrt{t_*}$. Define $\nu_J$ and $\PP_J$ as  in Lemma~\ref{lem:f_t}.
Note that the set
\[
A:=\Big\{ x \in \{1,-1\}^V: \, \sum_{v \in J} x_v \le 3\eps|J| +2\Big\}\,,
\]
satisfies $\nu_J(A) \ge 1/2$ as the mean and the median of the binomial distribution differ by at most one (see, e.g., \cite{kaas1980mean}).
For all $x \in A$,
$$\frac{d\nu_J}{d\nu}(x)=\prod_{v \in J}(1+3\eps x_v) \le \exp \Bigl(3\eps\sum_{v \in J} x_v \Bigr) \le 8e^{9\eps^2 |J|}\,. $$
Consider the event $A^*:=\{f_t(o) \ge \eps\} \cap \{f_0  \in A\} \,.$
Since $\frac{d\PP_{\!J}}{d\PP}(x,\cdot)=\frac{d\nu_{\!J}}{d\nu}(x)$ for all $x$, we can infer that
\begin{equation}
\PP _J(A^*)=\EE \Big[\frac{d\PP _J}{d\PP} \mathds 1_{A^*}\Big] \le 8e^{9\eps^2 |J|}\PP (A^*)\le 8e^{72\eps^2 \sqrt{t_*}}\PP (A^*).
\end{equation}
This finishes the proof of the proposition, since by Lemma~\ref{lem:f_t}, we have $\mathbb P _J(A^*) \ge \mathbb P_J (f_t(o)\ge \eps ) -\nu _J(A ^c) \ge 1/10$.
\end{proof}

\begin{prop}\label{theorem:lowerboundLp}
Consider the edge-averaging process on $G=(V,E)$ where $G$ is either the $n$-cycle or the infinite path $\mathbb Z$. For every $p\in[1,\infty)$ and $\delta>0$, there are i.i.d.\ random variables $\{f_0(v)\}_{v \in V}$ in $L^p$, such that if $t^*=\min (t,n^2)>30^4$, then  
\begin{equation}
\|f_t(o)-\mu\|_p\geq 
\begin{cases}
    t_*^{\frac{1-p}{2p}-\delta}&p\in [1,2)\,,\\
    t_*^{-1/4}&p \ge 2\,.
\end{cases}
\end{equation}

\end{prop}
\begin{proof}
First, let $\{f_0(v)\}_{v\in V}$ be as in Proposition~\ref{prop:lowlocal}. Applying that proposition with $\eps :=10t_*^{-1/4}$, we obtain 
$$\mathbb E [|f_t(o)|^p] \ge  \eps ^p\mathbb P (|f_t(o)|\ge \eps ) \ge c t_*^{-p/4},$$
for some universal constant $c>0$. Replacing $f_0$ by $c^{-1/p}f_0$ and using linearity of the map $f_0 \mapsto f_t$, establishes the case $p \ge 2$ of the proposition.

Now suppose that $p \in [1,2)$,
and let $\{f_0(v)\}_{v \in V}$
be symmetric stable variables with exponent $\alpha \in (p,2)$ (See, for example, \cite[Section~VI.1]{feller1991introduction}.)
Then $f_0(v) \in L^p$ and  $\mu=0$. Moreover, conditionally on $\mathcal F_\VarClock$, the law of
$f_t(o)=\sum_{v \in V} m_t(v) f_0(v)$
is the same as that of 
$\|m_t\|_{\alpha} \cdot f_0(o)$,
where $\|m_t\|_{\alpha}^\alpha=\sum_{v \in V} m_t ^\alpha(v)$.
By H\"older's inequality, for $J \subset V$,  
$$m_t(J)= \sum_{v \in J} m_t(v) \le\|m_t\|_{\alpha} \cdot |J|^\frac{\alpha-1}{\alpha} \,.$$
Using this inequality with $J$ as in Lemma \ref{lem:omega}, we obtain that on the event $\Omega$ defined there,
$\|m_t\|_{\alpha} \ge    c_\alpha  {t_*}^{\frac{1-\alpha}{2\alpha}}\,.$
Therefore,
$$\EE[|f_t(o)|^p \mathds 1_\Omega] \ge \PP(\Omega) \cdot c_\alpha^p
{t_*}^{\frac{p(1-\alpha)}{2\alpha}} \EE[|f_0(o)|^p] \,.$$ 
Applying lemma \ref{lem:omega} and scaling $f_0$, we get that
$$\EE[|f_t(o)|^p]^{1/p} \ge 
t_*^{\tfrac{1-\alpha}{2\alpha}} \,.$$
Taking $\alpha>p$ close to $p$ completes the proof.
\end{proof}

\section{Existence and uniqueness of the averaging process}\label{section:existence} 
When the graph $G$ is infinite, it is not clear that there is a unique process satisfying \eqref{equation definition}, since an infinite cascade of dependence might arise.
This can be related to a basic feature of the continuous-time random walk $(X_t)_{t\ge 0}$ on $G$, defined as follows. At time $t=0$, the particle is located at $X_0=o$. The edges are equipped with i.i.d.\ Poisson clocks of rate $1$. When the clock on an edge incident to the particle rings, the particle crosses this edge with probability $1/2$.

A graph $G$ is called \emph{stochastically complete} if this random walk is well defined for all $t\ge 0$. This property holds in finite graphs and in graphs of bounded degree, but there are graphs with sufficiently fast growth that are not stochastically complete, since $X_t$ escapes to infinity in finite time. A detailed study of stochastic completeness is given in \cite{MR3152724}.

When $G$ is stochastically complete, we can define the fragmentation process originating from $o$ by 
\begin{equation}
    m_t(v):=\mathbb P_o(X_t=v\mid \mathcal F _\VarClock ),
\end{equation}
where $\mathcal F _\VarClock$ is the $\sigma$-algebra generated by the Poisson clocks. Then we can define the edge-averaging process by Equation~\eqref{equation definition}.

\section{Open problems and future directions }\label{section:future} 
\begin{enumerate}
\item Consider the edge-averaging process on an infinite graph $G$. Does almost sure convergence of the opinions $f_t(v)$ hold under the natural assumption that the i.i.d.\ initial opinions $\{ f_0(v)\}_{v\in V}$ are in $L^1$?

This question, due to Gantert and Vilkas \cite{gantert2024averaging}, was already mentioned in the introduction. 

In Theorem~\ref{theorem infinite}, we proved almost sure convergence under the stronger assumption that $f_0(v)$ is in $L^p$ for some $p>4$.

\item A natural question is whether the results of this paper continue to hold if the independence assumption on the initial opinions is relaxed. Even in the one-dimensional lattice $\mathbb Z$, the opinions at a node may fail to converge for certain bounded initial opinion profiles, e.g.,
if
$f_0(j)=(-1)^k$ when $|j| \in [n_k,n_{k+1})$, where $\{n_k\}$ is a rapidly increasing sequence.
Such an example is analyzed in~\cite[Claim 6.7] {elboim2024asynchronous} for a slightly different model, the asynchronous DeGroot dynamics, but the same argument applies to the edge-averaging dynamics. 

Thus, some assumption on the initial opinions, such as correlation decay or ergodicity,  is needed. 
For example, the result in the case $p=2$ of Theorem~\ref{theorem:Lp}  holds if the initial opinions are uniformly bounded in $L^2$ with  the same mean, and satisfy a uniform bound on correlation sums:
\begin{equation} \label{sumcor}
 \sup_{v\in V} \sum_{w \in V}  
 \text{Cov}\big(f_t(v),f_t(w)\big) \le C <\infty \,.
\end{equation}
It would be interesting to determine which correlation bounds suffice for our other upper bounds to persist. 

In a different direction, perhaps almost sure convergence, as in Theorem~\ref{theorem infinite}, holds if the underlying graph is the lattice $\mathbb Z^d$ and the initial opinions are bounded, stationary, and ergodic; proving this would require significant new ideas.

\item Let $G$ be a connected graph  with $n$ nodes and maximal degree $\Delta$. Then it is not hard to obtain a $C_\eps(\Delta) \log n$ lower bound for the expected $\eps$-consensus time (this is the time that it takes until all the vertices update their opinions at least once). On which families of graphs is this upper bound tight (up to a constant factor)?
We expect a positive answer on expanders and high-dimensional grids.

\item  In Theorem \ref{theorem finite}, we estimate the consensus time on a connected  graph with $n$ nodes, for i.i.d.\ initial opinions that are bounded by a fixed constant.  How does the worst-case consensus time scale with $n$, if the boundedness assumption is relaxed to a bound on the 
$L^p$ norm ?
Theorem \ref{theorem:Lp} implies a power-law upper bound on the consensus time, but we do not have matching lower bounds. 

\end{enumerate}

\section{Acknowledgements}
We thank Ben Golub and Timo Vilkas for their insightful comments. We are grateful to Nina Gantert for fruitful discussions. We thank Yuval Peretz for writing a simulation of the averaging process and producing Figure~\ref{fig:simulation} and Figure~\ref{fig:charts}. RP acknowledges the support of the Israel Science Foundation Grant \#2566/20.

\bibliographystyle{plain}
\bibliography{pnas-sample}

\begin{thebibliography}{10}

\bibitem{MR2908618}
David Aldous and Daniel Lanoue.
\newblock A lecture on the averaging process.
\newblock {\em Probab. Surv.}, 9:90--102, 2012.

\bibitem{boyd2006randomized}
Stephen Boyd, Arpita Ghosh, Balaji Prabhakar, and Devavrat Shah.
\newblock Randomized gossip algorithms.
\newblock {\em IEEE transactions on information theory}, 52(6):2508--2530, 2006.

\bibitem{burkholder1988sharp}
Donald~L Burkholder.
\newblock Sharp inequalities for martingales and stochastic integrals.
\newblock {\em Ast{\'e}risque}, 157(158):75--94, 1988.

\bibitem{caputo2023cutoff}
Pietro Caputo, Matteo Quattropani, and Federico Sau.
\newblock Cutoff for the averaging process on the hypercube and complete bipartite graphs.
\newblock {\em Electronic Journal of Probability}, 28:1--31, 2023.

\bibitem{MR4385355}
Sourav Chatterjee, Persi Diaconis, Allan Sly, and Lingfu Zhang.
\newblock A phase transition for repeated averages.
\newblock {\em Ann. Probab.}, 50(1):1--17, 2022.

\bibitem{chow2012probability}
Yuan~Shih Chow and Henry Teicher.
\newblock {\em Probability theory: independence, interchangeability, martingales}.
\newblock Springer Science \& Business Media, 2012.

\bibitem{dolev1986reaching}
Danny Dolev, Nancy~A Lynch, Shlomit~S Pinter, Eugene~W Stark, and William~E Weihl.
\newblock Reaching approximate agreement in the presence of faults.
\newblock {\em Journal of the ACM (JACM)}, 33(3):499--516, 1986.

\bibitem{elboim2024asynchronous}
Dor Elboim, Yuval Peres, and Ron Peretz.
\newblock The asynchronous degroot dynamics.
\newblock {\em Random Structures \& Algorithms}, 65(4):857--895, 2024.

\bibitem{feller1991introduction}
William Feller.
\newblock {\em An introduction to probability theory and its applications, Volume 2}, volume~81.
\newblock John Wiley \& Sons, 1991.

\bibitem{MR3152724}
Matthew Folz.
\newblock Volume growth and stochastic completeness of graphs.
\newblock {\em Trans. Amer. Math. Soc.}, 366(4):2089--2119, 2014.

\bibitem{gantert2020strictly}
Nina Gantert, Markus Heydenreich, and Timo Hirscher.
\newblock {Strictly weak consensus in the uniform compass model on $\mathbb{Z}$}.
\newblock {\em Bernoulli}, 26(2):1269 -- 1293, 2020.

\bibitem{gantert2024averaging}
Nina Gantert and Timo Vilkas.
\newblock The averaging process on infinite graphs.
\newblock {\em arXiv preprint arXiv:2408.06859}, 2024.

\bibitem{gollin2024sharing}
J~Pascal Gollin, Kevin Hendrey, Hao Huang, Tony Huynh, Bojan Mohar, Sang-il Oum, Ningyuan Yang, Wei-Hsuan Yu, and Xuding Zhu.
\newblock Sharing tea on a graph.
\newblock {\em arXiv preprint arXiv:2405.15353}, 2024.

\bibitem{golub2010naive}
Benjamin Golub and Matthew~O Jackson.
\newblock Naive learning in social networks and the wisdom of crowds.
\newblock {\em American Economic Journal: Microeconomics}, 2(1):112--149, 2010.

\bibitem{kaas1980mean}
Rob Kaas and Jan~M Buhrman.
\newblock Mean, median and mode in binomial distributions.
\newblock {\em Statistica Neerlandica}, 34(1):13--18, 1980.

\bibitem{movassagh2024repeated}
Ramis Movassagh, Mario Szegedy, and Guanyang Wang.
\newblock Repeated averages on graphs.
\newblock {\em The Annals of Applied Probability}, 34(4):3781--3819, 2024.

\bibitem{Sadler2016-SADLIS}
Evan Sadler and Ben Golub.
\newblock Learning in social networks,.
\newblock In Bramoull\'e Yann, Andrea Galeotti, and Brian Rogers, editors, {\em The Oxford Handbook of the Economics of Networks}. Oxford University Press USA, 2016.

\bibitem{stein2011functional}
Elias~M Stein and Rami Shakarchi.
\newblock {\em Functional analysis: introduction to further topics in analysis}, volume~4.
\newblock Princeton University Press, 2011.

\end{thebibliography}

\end{document}